\documentclass[11pt,reqno]{amsart}
\usepackage{amsmath,amssymb}
\makeatletter \oddsidemargin.9375in \evensidemargin \oddsidemargin
\begin{document}
\begin{center}
\thispagestyle{empty} \setcounter{page}{1} {\large\bf Approximately
n--Jordan derivations: A fixed point approach \vskip.20in
{\Small \bf A. Ebadian} \\[2mm]

{\footnotesize  Department of Mathematics, Faculty of Science,
Urmia University, Urmia, Iran\\
e-mail: {\tt  ebadian.ali@gmail.com}}

 }
\end{center}
\vskip 5mm \noindent{\footnotesize{\bf Abstract.} Let $n\in \Bbb
N-\{1\},$ and let  $A$ be a Banach algebra. An additive map $D: A\to
A$ is called n-Jordan derivation if
$$D(a^n)=D(a)a^{n-1}+aD(a)a^{n-2}+...+a^{n-2}D(a)a+a^{n-1}D(a),$$
for all $a \in {A}$. Using fixed point methods, we investigate  the
stability of n--Jordan derivations (n--Jordan $^*-$derivations) on
Banach algebras ($C^*-$algebras). Also we show that  to each
approximate $^*-$Jordan derivation $f$ in a $C^*-$ algebra  there
corresponds a unique $^*-$derivation near to $f$.
 \vskip.10in
\footnotetext {2000 Mathematics Subject Classification. Primary
39B52; Secondary 39B82; 46HXX.}

\footnotetext {Keywords: Alternative fixed point; generalized
Hyers--Ulam stability; n--Jordan derivations.}
\vskip.10in
\newtheorem{df}{Definition}[section]
\newtheorem{rk}[df]{Remark}
\newtheorem{lem}[df]{Lemma}
\newtheorem{thm}[df]{Theorem}
\newtheorem{pro}[df]{Proposition}
\newtheorem{cor}[df]{Corollary}
\newtheorem{ex}[df]{Example}
\setcounter{section}{0} \numberwithin{equation}{section} \vskip .2in
\begin{center}
\section{Introduction}
\end{center}

Let $n\in \Bbb N-\{1\},$ and let  $A$ be an algebra. A linear map
$D: A\to A$ is called n--Jordan derivation  if
$$D(a^n)=D(a)a^{n-1}+aD(a)a^{n-2}+...+a^{n-2}D(a)a+a^{n-1}D(a),$$
for all $a \in {A}$.  A 2--Jordan derivation is a Jordan derivation,
in the
 usual  sense. A classical result of Herstein \cite{He} asserts that any Jordan
derivation on a prime ring with characteristic different from two is
a derivation. A brief proof of Herstein's result can be found in
1988 by Brešar and Vukman \cite{B-V}. Cusack \cite{C} generalized
Herstein's result to 2-torsion-free semiprime rings (see also
\cite{Bre} for an alternative proof). For some other results
concerning derivations on prime and semiprime rings, Jordan
derivations and n--Jordan derivations, we refer to \cite{Es1,
E-K-K-A, P, V, V-K}.

The stability of functional equations was first introduced  by S. M.
Ulam \cite{U} in 1940. More precisely, he proposed the following
problem: Given a group $G_1,$ a metric group $(G_2,d)$ and a
positive number $\epsilon$, does there exist a $\delta>0$ such that
if a function $f:G_1\longrightarrow G_2$ satisfies the inequality
$d(f(xy),f(x)f(y))<\delta$ for all $x,y\in G1,$ then there exists a
homomorphism $T:G_1\to G_2$ such that $d(f(x), T(x))<\epsilon$ for
all $x\in G_1?$ As mentioned above, when this problem has a
solution, we say that the homomorphisms from $G_1$ to $G_2$ are
stable. In 1941, D. H. Hyers \cite{H} gave a partial solution of
$Ulam^{,}s$ problem for the case of approximate additive mappings
under the assumption that $G_1$ and $G_2$ are Banach spaces. In
1950, Aoki \cite{Ao} generalized Hyers' theorem for approximately
additive mappings. In 1978, Th. M. Rassias \cite{R1} generalized the
theorem of Hyers by considering the stability problem with unbounded
Cauchy differences.

 According to
Th. M. Rassias theorem:
\begin{thm}\label{t1} Let $f:{E}\longrightarrow{E'}$ be a mapping from
 a norm vector space ${E}$
into a Banach space ${E'}$ subject to the inequality
$$\|f(x+y)-f(x)-f(y)\|\leq \epsilon (\|x\|^p+\|y\|^p)$$
for all $x,y\in E,$ where $\epsilon$ and p are constants with
$\epsilon>0$ and $p<1.$ Then there exists a unique additive mapping
$T:{E}\longrightarrow{E'}$ such that
$$\|f(x)-T(x)\|\leq \frac{2\epsilon}{2-2^p}\|x\|^p$$ for all $x\in E.$
If $p<0$ then inequality $(1.3)$ holds for all $x,y\neq 0$, and
$(1.4)$ for $x\neq 0.$ Also, if the function $t\mapsto f(tx)$ from
$\Bbb R$ into $E'$ is continuous for each fixed $x\in E,$ then T is
linear.
\end{thm}

On the other hand J. M. Rassias \cite {JR1} generalized the Hyers
stability result by presenting a weaker condition controlled by a
product of different powers of norms. If it is assumed that there
exist constants $\Theta\geq0$ and $p_1, p_2 \in \Bbb R$ such that
$p=p_1+p_2\neq 1,$ and $f:E\to E^{'}$ is a map from a norm space $E$
into a Banach space $E^{'}$ such that the inequality
$$\|f(x+y)-f(x)-f(y)\|\leq \Theta \|x\|^{p_1} \|y\|^{p_2}$$ for
all $x,y\in E,$ then there exists a unique additive mapping $T:E\to
E^{'}$ such that $$\|f(x)-T(x)\| \leq \frac{\Theta}{2-2^p}\|x\|^p
,$$ for all $x\in E.$ If in addition for every $x\in E,$ $f(tx)$ is
continuous in real $t$ for each
fixed $x,$ then $T$ is linear.\\

During the last decades several stability problems of functional
equations have been investigated by many mathematicians. A large
list of references concerning the stability of functional
equations can be found in \cite{Cz, H-I-R, J, JR, JR2, JR3, JR4, R2}.\\
Approximate derivations was first investigated by K.-W. Jun and
D.-W. Park \cite{J-P}. Recently, the stability of derivations have
been investigated by some authors; see \cite{ BAD2,J-K, J-P, P1} and
references therein.

On the other hand C$\breve{a}$dariu and Radu applied the fixed point
method to the investigation of the functional equations. (see also
\cite{C-R1, C-R2, C-R3, P-R, Ra, Ru}). Before proceeding to the main
results, we will state the following theorem.

\begin{thm}\label{t2}(The alternative of fixed point \cite{C-R}).
Suppose that we are given a complete generalized metric space
$(\Omega,d)$ and a strictly contractive mapping
$T:\Omega\rightarrow\Omega$ with Lipschitz constant $L$. Then for
each given $x\in\Omega$, either

$d(T^m x, T^{m+1} x)=\infty~$ for all $m\geq0,$\\
 or other exists a natural number $m_{0}$ such that\\

 $d(T^m x, T^{m+1} x)<\infty ~$for all $m \geq m_{0};$ \\

 the sequence $\{T^m x\}$ is convergent to a fixed point $y^*$ of $~T$;\\

$ y^*$is the unique fixed point of $~T$ in the
set $~\Lambda=\{y\in\Omega:d(T^{m_{0}} x, y)<\infty\};$\\

$ d(y,y^*)\leq\frac{1}{1-L}d(y, Ty)$ for all $~y\in\Lambda.$
\end{thm}

 In this paper, we will adopt the
fixed point alternative of C\u{a}dariu  and Radu to prove the
generalized Hyers--Ulam  stability of n--Jordan derivations
($^*-$n--Jordan derivations) on  Banach algebras ($C^*-$algebras)
associated with the following  Jensen--type functional equation
$$\mu f(\frac{x+y}{2})+\mu f(\frac{x-y}{2})= f(\mu x) ~~(\mu\in \Bbb T).$$

Throughout this paper assume that $A$ is a  Banach algebra.

 \vskip 5mm
%================================================================
\section{Main results}
%\setcounter{equation}{0}
%================================================================
By a following similar way as in \cite{P-R}, we obtain the next
theorem.

\begin{thm}\label{t1}
Let $f:A\to A$ be a mapping for which there exists a function
$\phi:A^3\to [0,\infty)$ such that

$$\|\mu f(\frac{x+y}{2})+\mu f(\frac{x-y}{2})-f(\mu x)+f( a^n)-
f(a)a^{n-1}+af(a)a^{n-2}$$ $$+...+a^{n-2}f(a)a+a^{n-1}f(a)\|\leq
\phi(x,y,a), \eqno (2.1)$$ for all $\mu \in \Bbb T$ and all $x,y,a
\in A.$ If  there exists an $L<1$ such that $\phi(x,y,a)\leq 2L
\phi(\frac{x}{2},\frac{y}{2},\frac{a}{2})$ for all $x,y,a\in A,$
then there exists a unique n--Jordan derivation $D:A\to A$ such that
$$\|f(x)-D(x)\| \leq \frac{L}{1-L}\phi(x,0,0)\eqno(2.2)$$
for all $x \in A.$
\end{thm}
\begin{proof}It follows from
$\phi(x,y,a)\leq 2L \phi(\frac{x}{2},\frac{y}{2},\frac{a}{2})$ that
$$lim_j 2^{-j}\phi(2^jx,2^jy,2^ja)=0 \eqno(2.3)$$ for all $x,y,a\in
A.$\\
Put $\mu =1, y=a=0$ in (2.1) to obtain
$$\|2f(\frac{x}{2})-f(x)\|\leq \phi(x,0,0)\eqno (2.4)$$
for all $x\in A.$ Hence,
$$\|\frac{1}{2}f(2x)-f(x)\|\leq \frac{1}{2} \phi(2x,0,0)\leq L\phi(x,0,0)\eqno (2.5)$$
for all $x\in A.$\\
Consider the set $X:=\{g\mid g:A\to B\}$ and introduce the
generalized metric on X:
$$d(h,g):=inf\{C\in \Bbb R^+:\|g(x)-h(x)\|\leq C\phi(x,0,0) \forall x\in A\}.$$
It is easy to show that $(X,d)$ is complete. Now we define  the
linear mapping $J:X\to X$ by $$J(h)(x)=\frac{1}{2}h(2x)$$ for all
$x\in A$. By Theorem 3.1 of \cite{C-R}, $$d(J(g),J(h))\leq Ld(g,h)$$
for
all $g,h\in X.$\\
It follows from (2.5) that  $$d(f,J(f))\leq L.$$ By Theorem 1.2, $J$
has a unique fixed point in the set $X_1:=\{h\in X: d(f,h)<
\infty\}$. Let $D$ be the fixed point of $J$. $D$ is the unique
mapping with
$$D(2x)=2D(x)$$ for all $x\in A$ satisfying there exists $C\in
(0,\infty)$ such that
$$\|D(x)-f(x)\|\leq C\phi(x,0,0)$$ for all $x\in A$. On the other hand we
have $lim_n d(J^n(f),D)=0$. It follows that
$$lim_n\frac{1}{2^n}{f(2^nx)}=D(x)\eqno (2.6)$$
for all $x\in A$. It follows from $d(f,D)\leq
\frac{1}{1-L}d(f,J(f)),$ that $$d(f,D)\leq \frac{L}{1-L}.$$ This
implies the inequality (2.2). It follows from (2.1), (2.3) and (2.6)
that
\begin{align*}\|&D(\frac{x+y}{2})+ D(\frac{x-y}{2})- D( x)\| \\
&=lim_n
\frac{1}{2^n}\|f(2^{n-1}(x+y))+f(2^{n-1}(x-y))-f(2^nx)\|\\
&\leq lim_n\frac{1}{2^n}\phi(2^nx,2^ny,0)\\
&=0
\end{align*}
for all $x,y \in A.$ So $$D(\frac{x+y}{2})+ D(\frac{x-y}{2})= D(
x)$$ for all $x,y \in A.$ Put $z=\frac{x+y}{2}, t=\frac{x-y}{2}$ in
above equation, we get
$$D(z)+D(t)=D(z+t)\eqno (2.7)$$ for all $z,t\in A.$
 Hence, $D$ is Cauchy additive. By putting $y=x, z=0$ in (2.1), we have
$$\|\mu f(\frac {2x}{2})-f(\mu x)\|\leq \phi(x,x,0)$$
for all $x\in A$. It follows that
$$\|D(2\mu x)-2\mu D(x)\|=lim_m
\frac{1}{2^m}\|f(2\mu 2^m x)-2\mu f(2^m x)\|\leq
lim_m\frac{1}{2^m}\phi(2^mx,2^mx,0)=0$$ for all $\mu \in \Bbb T$,
and all $x\in A.$ One can show that the mapping $D:A\to B$ is $\Bbb
C-$linear.  By putting $x=y=0$ in (2.1) it follows  that
\begin{align*}\|&D(a^n)-(D(a)a^{n-1}+aD(a)a^{n-2}+...+a^{n-2}D(a)a+a^{n-1}D(a))\|\\
&=lim_m\|\frac{1}{2^{mn}}f((2^{m}a)^n)-\frac{1}{2^{mn}}(f(2^{m}2^{m(n-1)}a)+f(2^{2m}2^{m(n-2)}a)\\
&+f(2^{3m}2^{m(n-3)}a))^n+...f(2^{m(n-1)}2^{m}a)\|\leq lim_m \frac{1}{2^{mn}}\phi(0,0,2^ma)\\
&\leq lim_m \frac{1}{2^m}\phi(0,0,2^ma)\\
&=0
\end{align*}
for all $a \in A.$  Thus $D:A\to A$ is an n--Jordan derivation
satisfying (2.2), as desired.
\end{proof}

Let $A$ be a $C^*-$algebra.  Note that an n--Jordan derivation
 $D:A\to A $ is an $^*-$n--Jordan
derivation if $D $ satisfies
$$D(a^*)=(D(a))^* $$ for all $a\in A.$

 We  establish  the generalized Hyers--Ulam
stability of $^*-$n--Jordan derivations  on $C^*-$algebras by using
the alternative fixed point theorem.

\begin{thm}\label{t31}
Let $f:A\to A$ be a mapping for which there exists a function
$\phi:A^4\to [0,\infty)$ satisfying
$$\|\mu f(\frac{x+y}{2})+\mu f(\frac{x-y}{2})-f(\mu x)+f( a^n)-
f(a)a^{n-1}+af(a)a^{n-2}$$
$$+...+a^{n-2}f(a)a+a^{n-1}f(a)+f(w^*)-(f(w))^*\|\leq\phi(x,y,a,w),\eqno(2.8)$$
for all $\mu \in \Bbb T$ and all $x,y,a,w \in A.$ If  there exists
an $L<1$ such that $$\phi(x,y,a,w)\leq 2L
\phi(\frac{x}{2},\frac{y}{2},\frac{a}{2},\frac{w}{2})$$ for all
$x,y,a,w\in A,$ then there exists a unique  $^*-$n--Jordan
derivation $D:A\to A$ such that
$$\|f(x)-D(x)\| \leq \frac{L}{1-L}\phi(x,0,0,0)\eqno(2.9)$$
for all $x \in A.$
\end{thm}
\begin{proof}By the same reasoning as the proof of Theorem 2.1, there exists
a unique n--Jordan derivation $D:A\to A$ satisfying (2.9). $D$ is
given by
$$D(x)=lim_n\frac{1}{2^n}{f(2^nx)}$$
for all $x\in A$. We have
\begin{align*}\|&D(w^*)-(D(w))^*\|\\
&=lim_n\|\frac{1}{2^{n}}f(2^nw^*)-\frac{1}{2^{n}}(f(2^nw))^*\|\\
&\leq lim_n \frac{1}{2^{n}}\phi(0,0,0,2^nw)\leq lim_n \frac{1}{2^n}\phi(0,0,0,2^nw)\\
&=0
\end{align*}
for all $w \in A.$  Thus $D:A\to A$ is $^*-$preserving. Hence, $D$
is an  $^*-$n--Jordan derivation satisfying (2.9), as desired.

\end{proof}

We prove the following Hyers--Ulam  stability problem for n--Jordan
derivations ($^*-$n--Jordan derivations) on Banach algebras
($C^*-$algebras).

\begin{cor}\label{t2}
Let $p\in (0,1), \theta \in [0,\infty)$ be real numbers. Suppose
$f:A \to B$ satisfies $$ \|\mu f(\frac{x+y}{2})+\mu
f(\frac{x-y}{2})-f(\mu x)+f( a^n)- f(a)a^{n-1}+af(a)a^{n-2}$$
$$+...+a^{n-2}f(a)a+a^{n-1}f(a)\| \leq
\theta(\|x\|^p+\|y\|^p+\|a\|^p),$$ for all $\mu \in \Bbb T$ and all
$x,y,a \in A.$  Then there exists a unique n--Jordan derivation
$D:A\to A$ such that
$$\|f(x)-D(x)\| \leq \frac{2^p\theta}{2-2^p}\|x\|^p$$
for all $x \in A.$
\end{cor}
\begin{proof}
It follows from Theorem 2.1, by putting
$\phi(x,y,a):=\theta(\|x\|^p+\|y\|^p+\|a\|^p)$ all $x,y,a \in A$ and
$L=2^{p-1}$.
\end{proof}

\begin{cor}\label{t2}
Let $A$ be a $C^*-$algebra, $p\in (0,1), \theta \in [0,\infty)$ be
real numbers. Suppose $f:A \to A$ satisfies $$ \|\mu
f(\frac{x+y}{2})+\mu f(\frac{x-y}{2})-f(\mu x)+f( a^n)-
f(a)a^{n-1}+af(a)a^{n-2}$$
$$+...+a^{n-2}f(a)a+a^{n-1}f(a)+f(w^*)-(f(w))^*\|$$ $$ \leq
\theta(\|x\|^p+\|y\|^p+\|a\|^p+\|w\|^p),$$ for all $\mu \in \Bbb T$
and all $x,y,a,w \in A.$ Then there exists a unique $^*-$n--Jordan
derivation $D:A\to A$ such that
$$\|f(x)-D(x)\| \leq \frac{2^p\theta}{2-2^p}\|x\|^p$$
for all $x \in A.$
\end{cor}
\begin{proof}
It follows from Theorem 2.2, by putting
$\phi(x,y,a,w):=\theta(\|x\|^p+\|y\|^p+\|a\|^p+\|w\|^p)$ all
$x,y,a,w \in A$ and $L=2^{p-1}$.
\end{proof}

\begin{thm}\label{t1}
Let $f:A\to A$ be an odd  mapping for which there exists a function
$\phi:A^3\to [0,\infty)$ such that

$$\|\mu f(\frac{x+y}{2})+\mu f(\frac{x-y}{2})-f(\mu x)+f( a^n)-
f(a)a^{n-1}+af(a)a^{n-2}$$ $$+...+a^{n-2}f(a)a+a^{n-1}f(a)\|\leq
\phi(x,y,a), \eqno (2.10)$$ for all $\mu \in \Bbb T$ and all $x,y,a
\in A.$ If  there exists an $L<1$ such that $\phi(x,3x,a)\leq 2L
\phi(\frac{x}{2},\frac{3x}{2},\frac{a}{2})$ for all $x,y,a\in A,$
then there exists a unique n--Jordan derivation $D:A\to A$ such that
$$\|f(x)-D(x)\| \leq \frac{1}{2-2L}\phi(x,3x,0)\eqno(2.11)$$
for all $x \in A.$
\end{thm}
\begin{proof}
Putting $\mu =1, y=3x, a=0$ in (2.10), it follows by oddness of $f$
that
$$\|f(2x)-2f(x)\|\leq \phi(x,3x,0)$$
for all $x\in A.$ Hence,
$$\|\frac{1}{2}f(2x)-f(x)\|\leq \frac{1}{2} \phi(x,3x,0)\leq L\phi(x,3x,0)\eqno (2.12)$$
for all $x\in A.$\\
Consider the set $X:=\{g\mid g:A\to B\}$ and introduce the
generalized metric on X:
$$d(h,g):=inf\{C\in \Bbb R^+:\|g(x)-h(x)\|\leq C\phi(x,0,0) \forall x\in A\}.$$
It is easy to show that $(X,d)$ is complete. Now we define  the
linear mapping $J:X\to X$ by $$J(h)(x)=\frac{1}{2}h(2x)$$ for all
$x\in A$. By Theorem 3.1 of \cite{C-R}, $$d(J(g),J(h))\leq Ld(g,h)$$
for
all $g,h\in X.$\\
It follows from (2.12) that  $$d(f,J(f))\leq L.$$ By Theorem 1.2,
$J$ has a unique fixed point in the set $X_1:=\{h\in X: d(f,h)<
\infty\}$. Let $D$ be the fixed point of $J$. $D$ is the unique
mapping with
$$D(2x)=2D(x)$$ for all $x\in A$ satisfying there exists $C\in
(0,\infty)$ such that
$$\|D(x)-f(x)\|\leq C\phi(x,3x,0)$$ for all $x\in A$. On the other hand we
have $lim_n d(J^n(f),D)=0$. It follows that
$$lim_n\frac{1}{2^n}{f(2^nx)}=D(x)$$
for all $x\in A$. It follows from $d(f,D)\leq
\frac{1}{1-L}d(f,J(f)),$ which implies that $$d(f,D)\leq
\frac{1}{2-2L}.$$ This implies the inequality (2.11).  The rest of
proof is similar to the proof of  Theorem 2.1.\end{proof}

\begin{cor}\label{t2}
Let $0<r<\frac{1}{2}, \theta \in [0,\infty)$ be real numbers.
Suppose $f:A \to A$ satisfies $$ \|\mu f(\frac{x+y}{2})+\mu
f(\frac{x-y}{2})-f(\mu x)+f( a^n)- f(a)a^{n-1}+af(a)a^{n-2}$$
$$+...+a^{n-2}f(a)a+a^{n-1}f(a)\| \leq
\theta(\|x\|^r\|y\|^r+\|a\|^{2r}),$$ for all $\mu \in \Bbb T$ and
all $x,y,a \in A.$  Then there exists a unique n--Jordan derivation
$D:A\to A$ such that
$$\|f(x)-D(x)\| \leq \frac{3^r\theta}{2-2^r}\|x\|^{2r}$$
for all $x \in A.$
\end{cor}
\begin{proof}
It follows from Theorem 2.5, by putting
$\phi(x,y,a):=\theta(\|x\|^r\|y\|^r+\|a\|^{2r})$ all $x,y,a \in A$
and $L=2^{2r-1}$.
\end{proof}

\begin{thm}\label{t31}
Let $f:A\to A$ be an odd mapping for which there exists a function
$\phi:A^4\to [0,\infty)$ satisfying
$$\|\mu f(\frac{x+y}{2})+\mu f(\frac{x-y}{2})-f(\mu x)+f( a^n)-
f(a)a^{n-1}+af(a)a^{n-2}$$
$$+...+a^{n-2}f(a)a+a^{n-1}f(a)+f(w^*)-(f(w))^*\|\leq\phi(x,y,a,w),$$
for all $\mu \in \Bbb T$ and all $x,y,a,w \in A.$ If  there exists
an $L<1$ such that $$\phi(x,3x,a,w)\leq 2L
\phi(\frac{x}{2},\frac{3x}{2},\frac{a}{2},\frac{w}{2})$$ for all
$x,a,w\in A,$ then there exists a unique  $^*-$n--Jordan derivation
$D:A\to A$ such that
$$\|f(x)-D(x)\| \leq \frac{1}{2-2L}\phi(x,3x,0,0)\eqno(2.13)$$
for all $x \in A.$
\end{thm}
\begin{proof}By the same reasoning as the proof of Theorem 2.1, there exists
a unique n--Jordan derivation $D:A\to A$ satisfying (2.13). $D$ is
given by
$$D(x)=lim_n\frac{1}{2^n}{f(2^nx)}$$
for all $x\in A$. We have
\begin{align*}\|&D(w^*)-(D(w))^*\|\\
&=lim_n\|\frac{1}{2^{n}}f(2^nw^*)-\frac{1}{2^{n}}(f(2^nw))^*\|\\
&\leq lim_n \frac{1}{2^{n}}\phi(0,0,0,2^nw)\leq lim_n \frac{1}{2^n}\phi(0,0,0,2^nw)\\
&=0
\end{align*}
for all $w \in A.$  Thus $D:A\to A$ is $^*-$preserving. Hence, $D$
is an  $^*-$n--Jordan derivation satisfying (2.13), as desired.

\end{proof}

\begin{cor}\label{t2}
Let $0<r<\frac{1}{2}, \theta \in [0,\infty)$ be real numbers.
Suppose $f:A \to A$ satisfies $$ \|\mu f(\frac{x+y}{2})+\mu
f(\frac{x-y}{2})-f(\mu x)+f( a^n)- f(a)a^{n-1}+af(a)a^{n-2}$$
$$+...+a^{n-2}f(a)a+a^{n-1}f(a)+f(w^*)-(f(w))^*\| \leq
\theta(\|x\|^r\|y\|^r+\|a\|^{2r}+\|w\|^r),$$ for all $\mu \in \Bbb
T$ and all $x,y,a,w \in A.$  Then there exists a unique
$^*-$n--Jordan derivation $D:A\to A$ such that
$$\|f(x)-D(x)\| \leq \frac{3^r\theta}{2-2^r}\|x\|^{2r}$$
for all $x \in A.$
\end{cor}
\begin{proof}
It follows from Theorem 2.7, by putting
$\phi(x,y,a):=\theta(\|x\|^r\|y\|^r+\|a\|^{2r})$ all $x,y,a \in A$
and $L=2^{2r-1}$.
\end{proof}

In 1996, Johnson  \cite{Jo} proved the following theorem (see also
Theorem 2.4 of \cite{H-L}).

\begin{thm}
Suppose $\mathcal A$ is a $C^*-$ algebra and $\mathcal M$ is a
Banach $\mathcal A-$module.  Then each Jordan derivation
 $d:\mathcal A\to \mathcal M$ is a derivation.
\end{thm}

Now,  we show that  to each approximate $^*-$Jordan derivation $f$
in a $C^*-$ algebra  there corresponds a unique $^*-$derivation near
to $f$.

\begin{cor}\label{t2}
Let $0<r<\frac{1}{2}, \theta \in [0,\infty)$ be real numbers.
Suppose $f:A \to A$ satisfies $$ \|\mu f(\frac{x+y}{2})+\mu
f(\frac{x-y}{2})-f(\mu x)+f( a^2)- f(a)a-af(a)+f(w^*)-(f(w))^*\|$$
$$
\leq \theta(\|x\|^r\|y\|^r+\|a\|^{2r}+\|w\|^r),$$ for all $\mu \in
\Bbb T$ and all $x,y,a,w \in A.$  Then there exists a unique
$^*-$Jordan derivation $D:A\to A$ such that
$$\|f(x)-D(x)\| \leq \frac{3^r\theta}{2-2^r}\|x\|^{2r}$$
for all $x \in A.$
\end{cor}
\begin{proof}
It follows from Theorem 2.9 and  Corollary 2.8.
\end{proof}

{\small
%----------------------------------------------------------------------%

}
\end{document}